\documentclass[a4paper,10pt]{amsart}

\usepackage[utf8]{inputenc}   

\usepackage{amsfonts}
\usepackage{amsthm}
\usepackage{amssymb}
\usepackage{amsmath}
\usepackage{enumerate}
\usepackage{tikz}
\usetikzlibrary{patterns}
\usepackage{url}
\usepackage{nameref} 
\usepackage{hyperref} 
\usepackage{color}

\newtheorem{theorem}{Theorem}

\newtheorem{lemma}{Lemma}

\newcommand{\FT}[0]{\mathbb{F}_2}

\def\en{
        \begin{pmatrix}
                0      \\
                \vdots \\
                0      \\
                1
        \end{pmatrix}
        }

\def\ones{
        \begin{pmatrix}
                1      \\
                \vdots \\
                1      \\
                1
        \end{pmatrix}
        }

\begin{document}
\title{Low discrepancy sequences  failing  Poissonian  pair correlations}

\author{Ver\'onica Becher \and Olivier Carton \and Ignacio Mollo Cunningham}

\date{February 28, 2019}

\begin{abstract}
  M. Levin defined a real number $x$ that satisfies that the sequence of
  the fractional parts of $(2^n x)_{n\geq 1}$ are such that the first $N$
  terms have discrepancy $O((\log N)^2/ N)$, which is the smallest
  discrepancy known for this kind of parametric sequences.  In this work we
  show that the fractional parts of the sequence $(2^n x)_{n\geq 1}$ fail
  to have Poissonian pair correlations.  Moreover, we show that all the
  real numbers $x$ that are variants of Levin's number using Pascal
  triangle matrices are such that the fractional parts of the sequence
  $(2^n x)_{n\geq 1}$ fail to have Poissonian pair correlations.
\end{abstract}
\bigskip
\maketitle


\section{Introduction and statement of results}

A sequence $(x_n)_{n \geq 1}$ of real numbers in the unit interval is said
to have \emph{Poissonian pair correlations} if for all non-negative real
numbers $s$,
\begin{displaymath}
  \lim_{N\to\infty}F_N(s) = 2s
\end{displaymath}
where
\begin{displaymath}
      F_N(s)=\frac{1}{N}\#\left\{\,(i,j)\,: \,1\leq i\neq j\leq N \text{ and } \Vert x_i - x_j \Vert < \frac{s}{N} \, \right\}. \label{pairs}
\end{displaymath}
and $\Vert x \Vert$ is the distance between $x$ and its nearest integer.
The function $F_N(s)$ counts the number of pairs $(x_n,x_m)$ for $1\leq m,n
\leq N$, $m\not = n$, of points which are within distance at most $s/N$ of
each other, in the sense of distance on the torus.  If $\lim_{N\to \infty}
F(s)=2s$ for all $s\geq 0$ then the asymptotic distribution of the pair
correlations of the sequence is Poissonian, and this explains that the
property is referred as having Poissonian pair correlations.  Almost surely
a sequence of independent identically distributed random variables in the
unit interval has this property.  Several particular sequences have been
proved to have the property, for example $(\sqrt n \mod 1)_{n\geq
  1}$~\cite{ElBazMarklofVinograd}.  It is known that for almost all real
numbers~$x$ $(a_n x \mod 1)_{n\geq 1}$ has the property when $a_n$ is
integer valued and $(a_n)_{n\geq 1}$ is
lacunary~\cite{RudnickZaharescu1999}; also when $a_n= n^2$ (or a higher
polynomial)~\cite{RudnickSarnackZaharescu,RudnickSarnak}.  However, for
specific values such as $x=\sqrt{2}$ and $a_n= n^2$ it is not known whether
$(a_n x \mod 1)_{n\geq 1}$ has Poissonian pair correlations or~not.

The property of Poissonian pair correlations implies uniform distribution
modulo~$1$, this was only recently proved in \cite[Theorem~1]{ALP2018} and
also in \cite[Corollary~1.2]{GrepstadLarcher}.  The converse does not
always hold.  Several uniformly distributed sequences of the form $(b^n x
\mod 1)_{n\geq 1}$ where $b$ is an integer greater than~$1$ and $x$ is a
constant were proved to fail the property of Poissonian pair correlations.
Pirsic and Stockinger~\cite{PirsicStockinger} proved it for Champernowne's
constant (defined in base~$b$).  Larcher and
Stockinger~\cite{LarcherStockinger} proved it for $x$ a Stoneham
number~\cite{Stoneham1973} and for every real number $x$ having an
expansion which is an infinite de Bruijn word~(see
\cite{BecherHeiber,Ugalde} for the presentation of these infinite words).
Larcher and Stockinger in~\cite{LS} also show the failure of the property
for other sequences of the form $(a_n x \mod 1)_{n\geq 1 }$.

In this paper we show that the sequence $(2^n \lambda \mod 1)_{n\geq 1}$,
where $\lambda$ is the real number defined by Levin in \cite[Theorem
2]{Levin1999}, fails to have Poissonian pair correlations.  Levin's number
$\lambda$ is defined constructively using Pascal triangle matrices and
satisfies that the discrepancy of the first $N$ terms of the sequence $(2^n
\lambda \mod 1)_{n\geq 1}$ is $O((\log N)^2/N)$. This is the smallest
discrepancy bound known for sequences of the form $(2^n x \mod 1)_{n\geq
  1}$ for some real number~$x$.

We also show that each of the real numbers $\rho$ considered by Becher and
Carton in~\cite{BecherCarton2018} are such that the sequence $(2^n \rho\mod
1)_{n\geq 1}$ fails to have Poissonian pair correlations.  These numbers
$\rho$ are variants of Levin's number $\lambda$ because they are defined
using rotations of Pascal triangle matrices and the sequence $(2^n \rho\mod
1)_{n\geq 1}$ has the same low discrepancy as that obtained by Levin.

\subsection{Levin's number}

We start by defining the number $\lambda$ given by Levin in~\cite[Theorem
2]{Levin1999} and further examined in~\cite{BecherCarton2018}.
As usual, we write $\FT$ to denote the field of two elements. In this work, 
we freely make the identification between binary words and vectors on $\FT$.
 We define recursively a sequence of matrices on $\FT$: 

\begin{displaymath}
M_0 = (1)
\quad\text{and for every $d\geq 0$},\quad
M_{d+1} = \begin{pmatrix} M_d & M_d \\ 0  & M_d \end{pmatrix}.
\end{displaymath}

The first elements of this sequence, for example, are:
\begin{displaymath}
        M_0 = (1) 
        \quad 
        M_1 = 
        \begin{pmatrix}
        1 & 1 \\ 0 & 1
        \end{pmatrix}
        \quad
        M_2 =
        \begin{pmatrix}
        1 & 1 & 1 & 1 \\
        0 & 1 & 0 & 1 \\
        0 & 0 & 1 & 1 \\
        0 & 0 & 0 & 1
        \end{pmatrix}.
\end{displaymath}

Let $d$ be a non-negative integer and let $e=2^d$.  The matrix $M_d \in
\FT^{e\times e}$ is upper triangular with $1$s on the diagonal, hence it is
non-singular.  Then, if 
\begin{displaymath}
  w_0,\dots,w_{2^e-1}
\end{displaymath}
is the enumeration of all vectors of length~$e$ in lexicographical order,
the sequence
\begin{displaymath}
  M_dw_0,\ldots,M_dw_{2^e-1}
\end{displaymath}
ranges over all vectors of length~$e$. We obtain the $d$-th block
of~$\lambda$ by concatenation of the terms of that sequence:
\begin{displaymath}
  \lambda_d = (M_dw_0)(M_dw_1)\dots(M_dw_{2^e-1}) 
\end{displaymath} 
Levin's constant $\lambda$ is  defined as the infinite concatenation
\begin{displaymath}
  \lambda = \lambda_0\lambda_1\lambda_2\dots
\end{displaymath}
The expansion of $\lambda$ in base $2$ starts as follows (the spaces are
just for convenience):
\begin{displaymath}
  {\tiny
\underbrace{01}_{\lambda_0} \ \underbrace{00\, 11\, 10\, 01}_{\lambda_1}\  
\underbrace{0000\, 1111\, 1010\, 0101\, 1100\, 0011\, 0110\, 1001\, 1000\, 0111\, 0010\, 1101\,  0100\, 1011\,  1110\, 0001}_{\lambda_2} \ 
\underbrace{00000000\, 11111111\ldots}_{\lambda_3}}
\end{displaymath}
   
Now we introduce a family $\mathcal{L}$ of constants which have similar
properties to those of~$\lambda$.  Let $\sigma$ be the rotation that takes
a word and moves its last letter at the beginning: that is,
$\sigma(a_1\dots a_n) = a_na_1\dots a_{n-1}$. We are going to use $\sigma$
to define a family of matrices obtained by selectively rotating some of the
columns of~$M_d$.

As before, assume $d$ is a non-negative integer and let $e=2^d$. We say
that a tuple $\nu=(n_1,\dots,n_e)$ of non-negative integers is
\emph{suitable} if
\begin{displaymath}
  n_e= 0 \mbox{ and } n_{i+1}\le n_i \le n_{i+1}+1 \mbox{ for each }1\le i\le e-1.
\end{displaymath} 
Let $C_1,\dots C_e$ denote the columns of $M_d$. Then, define
\begin{displaymath}
        M_d^\nu\,
                =\,
                \left(
                \sigma^{n_1} \left(C_1\right),
                \dots,
                \sigma^{n_e} \left(C_e\right)
                \right)
\end{displaymath}
For example, by taking $d=2$ we have $8$ different possible matrices, one
for every choice of $\nu$:

\begin{displaymath}
\begin{array}{cccc} 
M_2^{(0,0,0,0)} & M_2^{(1,0,0,0)} & M_2^{(1,1,0,0)} & M_2^{(2,1,0,0)} \\
\begin{pmatrix}
1 & 1 & 1 & 1 \\
0 & 1 & 0 & 1 \\
0 & 0 & 1 & 1 \\
0 & 0 & 0 & 1
\end{pmatrix}
& 
\begin{pmatrix}
0 & 1 & 1 & 1 \\
1 & 1 & 0 & 1 \\
0 & 0 & 1 & 1 \\
0 & 0 & 0 & 1
\end{pmatrix} 
& 
\begin{pmatrix}
0 & 0 & 1 & 1 \\
1 & 1 & 0 & 1 \\
0 & 1 & 1 & 1 \\
0 & 0 & 0 & 1
\end{pmatrix}
& 
\begin{pmatrix}
0 & 0 & 1 & 1 \\
0 & 1 & 0 & 1 \\
1 & 1 & 1 & 1 \\
0 & 0 & 0 & 1
\end{pmatrix} \\[10mm]
M_2^{(1,1,1,0)} & M_2^{(2,1,1,0)} & M_2^{(2,2,1,0)} & M_2^{(3,2,1,0)} \\
\begin{pmatrix}
0 & 0 & 0 & 1 \\
1 & 1 & 1 & 1 \\
0 & 1 & 0 & 1 \\
0 & 0 & 1 & 1
\end{pmatrix}
& 
\begin{pmatrix}
0 & 0 & 0 & 1 \\
0 & 1 & 1 & 1 \\
1 & 1 & 0 & 1 \\
0 & 0 & 1 & 1
\end{pmatrix} 
& 
\begin{pmatrix}
0 & 0 & 0 & 1 \\
0 & 0 & 1 & 1 \\
1 & 1 & 0 & 1 \\
0 & 1 & 1 & 1
\end{pmatrix}
& 
\begin{pmatrix}
0 & 0 & 0 & 1 \\
0 & 0 & 1 & 1 \\
0 & 1 & 0 & 1 \\
1 & 1 & 1 & 1
\end{pmatrix}
\end{array}
\end{displaymath}



As before, we let $e=2^d$ and $w_0,\dots,w_{2^e-1}$ be the increasingly
ordered sequence of all vectors in $\FT^e$.  We say that a word is an
$e$-\emph{affine necklace} if it can be written as the concatenation
$(Mw_0')(Mw_1')\dots(Mw_{2^e-1}')$ for some $z\in\FT^e$ and a suitable
tuple~$\nu$ with $M=M_d^\nu$ and $w_i'=w_i+z$ for $0 \le i \le 2^e-1$.

Finally, we define $\mathcal{L}$ to be the set of all binary words that can
be written as an infinite concatenation $\rho_0\rho_1\rho_2\dots$ where
every $\rho_d$ is a $2^d$-affine necklace. Note that
$\lambda\in\mathcal{L}$, by taking $\nu = (0,\ldots,0)$ everywhere.

The rest of this note is devoted to proving the following result:
\begin{theorem}\label{thm:main}
  For all $\rho\in\mathcal{L}$, the sequence of fractional parts of $(2^n
  \rho)_{n\geq 1}$ does not have Poissonian pair correlations.
\end{theorem}

\section{Lemmas}

Now we prove some necessary results. We present in an alternating manner
results about $M_d$ and its corresponding generalizations to the family of
matrices~$M_d^\nu$.
\begin{lemma}\label{bf:inversible}
  For all $d$, $M_d$ is triangular and all entries in its diagonal are
  ones. In particular, $M_d$ is non singular.
\end{lemma}
\begin{proof}
  This is easily proven with induction. $M_0=(1)$ satisfies the lemma, and
  if $M_d$ satisfies it, then $M_{d+1}=\begin{pmatrix} M_d & M_d \\ 0 & M_d
  \end{pmatrix}$ satisfies it too.
\end{proof}

\begin{lemma}\label{bf:inversible'}
  For all non-negative $d$ and for every suitable tuple $\nu$, $M_d^\nu$ is
  non-singular.
\end{lemma}
\begin{proof}
 This fact is proven in~\cite[Lemma 4]{BecherCarton2018}.
\end{proof}

\begin{lemma}\label{bf:opposites}
  For all $d$ and for all even $n$, $M_dw_n$ and $M_dw_{n+1}$ are
  complementary vectors. That is, the $i$-th coordinate of $M_dw_n$ equals
  zero if and only if the $i$-th coordinate of $M_dw_{n+1}$ equals one.
\end{lemma}
\begin{proof}
  The sequence $w_0, w_2, \ldots w_{2^e-1}$ is lexicographically ordered and
  hence the last entry of $w_n$ is zero whenever $n$ is even.  Therefore,
  $w_{n+1}$ only differs from $w_n$ in the last entry
  \begin{displaymath}
    M_dw_{n+1} = M_dw_n+M_d\en = (M_dw_n)+\ones.
  \end{displaymath}
\end{proof}

To simplify notation, from now on we write~$\overline{z}$ for the
complementary vector of~$z$.  Note that Lemma~\ref{bf:opposites} implies
that $\lambda_d$ can be written as a concatenation of words of the form
$w\overline{w}$.

\begin{lemma}\label{bf:opposites'}
  For all non-negative $d$, for all even $n$, and for every suitable tuple
  $\nu$, the vectors $M_d^\nu w_n$ and $M_d^\nu w_{n+1}$ are complementary.
\end{lemma}
\begin{proof}
  The last coordinate of $\nu$ is zero by definition. Therefore, the last
  column of $M_d^\nu$ is the same as the last column of $M_d$; that is,
  it is the vector of ones. The same argument used to prove
  Lemma~\ref{bf:opposites} applies.
\end{proof}

We say that a vector is \emph{even} if its last entry is $0$.  Hence, when
$n$ is even, $w_n$ is an even vector.

\begin{lemma}\label{bf:even-invariance}
  Let $d$ be a non-negative integer and $e=2^d$. The subspace of all even
  vectors of length $e$,
  \begin{displaymath}
    \mathbb{P}=\{v\in\FT^{e} \mid v_{e}=0 \}
  \end{displaymath}
  is invariant under $M_d$. Furthermore, $M_dw_n$ is an even vector if and
  only if $w_n$ is an even vector.
\end{lemma}
\begin{proof}
  By Lemma~\ref{bf:inversible}, $M_d$ is upper triangular and its diagonal
  is comprised by ones.  This implies that all of its columns except the
  last are even vectors.  Therefore, the only way to obtain an odd vector
  via the computation $M_dw$ is that $w$ itself is odd.
\end{proof}

\begin{lemma}\label{bf:even-invariance'}
  Let $d$ be a non-negative integer and $e=2^d$. Depending on $\nu$, there
  are two distinct possibilities:
  \begin{enumerate}
  \item The subspace of even vectors $\mathbb{P}$ is invariant under
    $M_d^\nu$. In this case, $w$ is an even vector if and only if
    $M_d^\nu w$ is an even vector.
  \item The subspace $\mathbb{P}$ is in bijection with the subspace
    $\{ (v_1,\dots,v_e)\in\mathbb{F}_2^e \mid v_1=0 \}$ via
    $M_d^\nu$. In this case, $w$ is an even vector if and only if
    $M_d^\nu w$ has a zero in its first coordinate.
  \end{enumerate}
\end{lemma}
\begin{proof}
  Take $\nu\,=\,(n_1,\dots n_e)$ such that $n_e=0$ and $n_{i+1}\leq n_i
  \leq n_{i+1}+1$ for all $1\leq i \leq e-1$. Combining $n_i \leq
  n_{i+1}+1$ for $i = e-1$ and $n_e = 0$ gives that $n_{e-1}$ is either
  zero or one.  We consider both possibilities separately.
  
  First, suppose that $n_{e-1}$ equals one. Then, all entries of $\nu$
  before it must be greater than one. That means that, when building
  $M_d^\nu$ from $M_d$, all of its columns except the last are rotated at
  least one position. Fix an index $i$ such that $1\leq i\leq e-1$, and
  consider $c$ the $i$-th column of the matrix $M_d$. We show that the
  first element of $\sigma^{n_i}(c)$ is zero.
        
  By Lemma~\ref{bf:inversible}, we know that $M_d$ is triangular. That
  means that the elements $c_{i+1},\dots, c_e$ are necessarily zeros. But
  the first element of $\sigma^{n_i}(c)$ is $c_{e-n_i+1}$; and from the
  inequality $n_i\leq e-i$ it follows that $e-n_i+1$ is greater or equal
  than $i+1$. Therefore, the first element of $\sigma^{n_i}(c)$ is zero.
        
  Because $i$ is any index between $1$ and $e-1$, it follows that the first
  $e-1$ columns of $M_d^\nu$ have a zero as their first coordinate. If $w$
  is an even vector, $M_dw$ is a linear combination of vectors which start
  with zero, and therefore $M_dw$ also starts with a zero. Conversely, if
  $w$ begins with a one then $M_dw$ must be an odd vector, because it's a
  linear combination of elements that start with a zero and the last column
  of $M_d^\nu$, which is the vector of ones. We conclude that $w$ is an
  even vector if and only if $M_d^\nu w$ starts with a zero.
        
  The case where $n_{e-1}$ equals zero is analogous, and we give an outline
  of the proof. First, prove that the first $e-1$ columns of $M_d^\nu$ are
  even vectors. Then, $M_d^\nu w$ is even if and only if $w$ is even.
\end{proof}

\section{Proof of the main theorem}

We first prove the theorem for $\lambda$ and at the end we explain how to
generalize the result to each number in the family $\mathcal{L}$. For any
given non-negative $d$, we set $e=2^d$ and show that for an appropriate
choice of increasing $N$ which depends on $d$ and $e$, $F_N(2)$ diverges.
For this we show that some selected patterns have too many occurrences in
$\lambda_d$.  More precisely, we count occurrences of binary words of
length $d+e$,

\begin{displaymath}
  a = a_1a_2\dots a_{d+e} 
\end{displaymath}  
such that
\begin{displaymath}
  \overline{a_1\dots a_d} = a_{e+1}\dots a_{e+d}. 
\end{displaymath}
The reason for this choice comes from Lemma~\ref{bf:opposites} and it will
soon become clear.  We need some terminology.  Given a word $a$ as above
and an occurrence of $a$ in $\lambda_d$,
\begin{enumerate}
\item let $k$ be  the number of zeros in $a_d\dots a_e$;
\item let $n$ be  the index  such that the $a$ occurs in $M_d w_n$;
\item let $z$ be the position in~$a$ that matches with the $e$-th (that is
  the last) symbol of $M_dw_n$.
\end{enumerate}

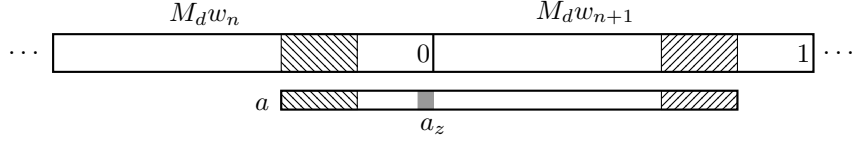
\begin{figure}[h]
  \begin{center}
    \begin{tikzpicture}
      \draw[thick] (0,0) rectangle (10,0.5);
      \draw[thick] (5,0)--(5,0.5);
      \draw[thin, pattern=north west lines] (3,0) rectangle (4,0.5);
      \draw[thin, pattern=north east lines] (8,0) rectangle (9,0.5);
      \node [left] at (0,0.25) {$\dots$};
      \node [right] at (10,0.25) {$\dots$};
      \node [anchor=base] at (4.87,0.125) {$0$};
      \node [anchor=base] at (9.87,0.125) {$1$};
      \node [above] at (2,0.5) {$M_dw_n$};
      \node [above] at (7,0.5) {$M_dw_{n+1}$};
      \filldraw[black!40!white] (4.8,-0.5) rectangle (5,-0.25);
      \draw[thick] (3,-0.5) rectangle (9,-0.25);
      \draw[thin, pattern=north west lines] (3,-0.5) rectangle (4,-0.25);
      \draw[thin, pattern=north east lines] (8,-0.5) rectangle (9, -0.25);
      \node at (2.75,-0.5) [anchor=base] {$a$};
      \node [below] at (5,-0.5) {$a_z$};
      \end{tikzpicture}
  \end{center}
  \caption{An occurrence of $a$}
  \label{fig:matchWithinL}
\end{figure}

We require $n$ to be an even number and $z$ to be in the range $d\leq z\leq
e$.  The latter is to prevent the word $a$ from spanning over more
than two words, and the former is to ensure that a match for the first $d$
letters automatically yields a match for the last $d$ letters (a
combination of Lemma~\ref{bf:opposites} and the hypothesis over $a$). In
addition, by Lemma~\ref{bf:even-invariance} we know that $M_dw_n$ is an
even word, and therefore $a_z$ must be zero.

We fix $k$ and count all possible occurrences in $\lambda_d$ of every possible
word $a$.  There are exactly
\begin{displaymath}
  2^{d-1}\binom{e-d+1}{k}
\end{displaymath}
words $a$ with $k$ zeros in $a_d\dots a_e$. For every one of them, we have
a choice of $k$ different $z$, because we know that $a_z$ must be zero.  We
claim that each of those choices for $z$ correspond to an actual occurrence
of $a$ in $\lambda_d$.  Let's suppose that the binary word $a$, whose length is
$d+e$, starts in $M_dw_n$ and continues in $M_dw_{n+1}$. Then, it must hold
that, for some $z$,
\begin{displaymath}
  a_1\dots a_z = (M_dw_n)_{e-z+1}\dots (M_dw_n)_e
\end{displaymath}
and  
\begin{displaymath}
  a_{z+1}\dots a_e = (M_dw_{n+1})_1\dots (M_dw_{n+1})_{e-z}.
\end{displaymath} 
So, Lemma~\ref{bf:opposites} allows us to conclude that
\begin{displaymath}
  M_dw_n = \overline{a_{z+1}\dots a_e}a_1\dots a_z.
\end{displaymath}

\begin{figure}[t]
        \centering
        \begin{tikzpicture}
        \filldraw[black!60!white] (0,0) rectangle (3,0.5);
        \filldraw[black!30!white] (4,0) rectangle (5,0.5);
        \filldraw[black!30!white] (5,0) rectangle (8,0.5);
        \draw[thick] (0,0) rectangle (10,0.5);
        \draw[thick] (5,0)--(5,0.5);
        \draw[thin, pattern=north west lines] (3,0) rectangle (4,0.5);
        \draw[thin, pattern=north east lines] (8,0) rectangle (9,0.5);
        \node [left] at (0,0.25) {$\dots$};
        \node [right] at (10,0.25) {$\dots$};
        \node [above] at (3.3,0.5) {$M_dw_n$};
        \node [above] at (8.3,0.5) {$M_dw_{n+1}$};
        
        \filldraw[black!30!white] (4,-0.90) rectangle (5,-0.65);
        \filldraw[black!30!white] (5,-0.90) rectangle (8,-0.65);
        \draw[thick] (3,-0.90) rectangle (9,-0.65);
        \draw[thin] (5,-0.90) -- (5,-0.65);
        \draw[thin, pattern=north west lines] (3,-0.90) rectangle (4,-0.65);
        \draw[thin, pattern=north east lines] (8,-0.90) rectangle (9, -0.65);
        \node at (2.75,-0.90) [anchor=base] {$a$};
        
        \draw [->,thick,dotted] (4.5,-0.60)--(4.5,-0.1);
        \draw [->,thick,dotted] (6.5,-0.60)--(6.5,-0.1);
        \draw [->,thick,dotted] (6.5,0.6) to [out=160,in=20] (1.5,0.6);
        
        \end{tikzpicture}
        \caption{Given an $a$ and a choice of $z$, 
        it is possible to find an unique position within $\lambda_d$ where the word $a$ occurs with alignment~$z$.}
        \label{fig:theMatchDeterminesN}
\end{figure}
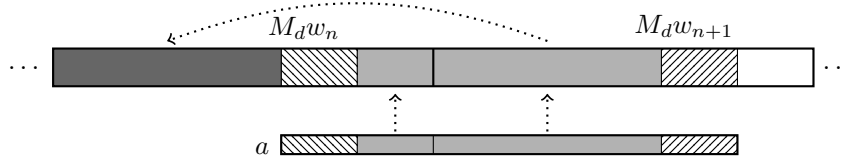

By Lemma~\ref{bf:inversible}, we know that there exists some $w_n$ that
satisfies this equation and by Lemma~\ref{bf:even-invariance}, $n$ must be
an even number.  Therefore, given a choice of $z$ there is an occurrence of
$a$.  We conclude that for every choice of $k$, and for every word $a\in
\{0,1\}^{e+d}$ with $k$ zeros in $a_d\dots a_e$, we have exactly $k$
occurrences within~$\lambda_d$.

We now prove that the sequence of fractional parts of $(2^n \lambda)_{n\geq 1}$
does not have Poissonian pair correlations. Take $s=2$ and $N=2^{d+e+1}$.
We prove that $\lim_{d\to\infty}F_N(s) = \infty$.  In order to do that, we
note that two different occurrences of the same word $a\in\{0,1\}^{e+d}$
correspond to two different suffixes of $\lambda$ that share its first
$e+d$ digits.

We write $\{x\}$ to denote $x-\lfloor x\rfloor$, the fractional expansion
of $x$.  If $a$ has two different occurrences within $\lambda$ at positions
$i$ and $j$ then

\begin{align*}
  \Vert \left\{2^i \lambda\right\}-\left\{2^j \lambda\right\} \Vert
  & =  \Vert 0.a_1\dots a_{e+d}\lambda_{i+d+e+1}\dots - 0.a_1\dots a_{e+d}\lambda_{j+d+e+1}\dots \Vert \\
  &  \leq \vert 0.a_1\dots a_{e+d}\lambda_{i+d+e+1}\dots - 0.a_1\dots a_{e+d}\lambda_{j+d+e+1}\dots \vert \\
  & <  2^{-\left(e+d\right)} \\
  & = \frac{s}{N}.\\
\end{align*} 
Therefore, if $i$ and $j$ are both no greater than $N$, the pairs $(i,j)$
and $(j,i)$ count for $F_N(s)$.  For indices $i$ and $j$ of $\lambda$ which
correspond to elements of $\lambda_d$ this is the case:
\begin{equation*}
   \vert \lambda_0\dots \lambda_d \vert  \ =\  \sum_{i=0}^{d} \vert \lambda_i \vert 
                                 \  =\  \sum_{i=0}^{d} 2^i 2^{2^i} 
                                 \  =\  \sum_{i=0}^{d} 2^{i+2^i} 
                                 \ <\  2^{2^d+d+1} 
                                 \  =\  N.
\end{equation*} 
Then, we are able to give a lower bound for $F_N(s)$, by taking all
possible pairs of occurrences of every word $a$ which satisfies the
condition $\overline{a_1\dots a_d} = a_{e+1}\dots a_{e+d}$:

\begin{align*}
   F_N(2) & \geq \, \frac{1}{N}\sum_{k=0}^{e-d+1} 2\left(2^{d-1}\binom{e-d+1}{k}\binom{k}{2}\right) \\
          & =  \frac{1}{2^{e+1}} \, \sum_{k=0}^{e-d+1} \binom{e-d+1}{k}\binom{k}{2} \\
          & = \, \frac{1}{2^{e+1}} \, \binom{e-d+1}{2}\,2^{e-d+1-2} \\
         & = \frac{1}{8e}\,(e-d+1)(e-d)
\end{align*} 
In the third step, we used the identity
\begin{displaymath}
  \sum_{k=0}^{n} \binom{n}{k}\binom{k}{2} \,=\, 2^{n-2}\binom{n}{2} 
\end{displaymath}
with $n=(e-d+1)$.  Thus,
\begin{displaymath}
  F_N(2)\,\geq\,\frac{(e-d+1)(e-d)}{8e}
\end{displaymath}
and the last expression diverges as $d\to\infty$ because $e$ is squared in
the numerator but linear in the denominator and $d$ is insignificant with
respect to $e$.  This concludes the proof that the sequence of the
fractional parts of $(2^n\lambda)_{n\geq1}$ does not have Poissonian pair
correlations.

We now explain how the proof extends to for any given constant in
$\mathcal{L}$.  Take $\rho\in\mathcal{L}$. Then $\rho$ can be written as a
concatenation
\begin{displaymath}
  \rho=\rho_0\rho_1\rho_2\dots
\end{displaymath} 
where each $\rho_d$ is a $2^d$-affine necklace.  That means that for every
$d$, there exists a suitable tuple $\nu$ such that
\begin{displaymath}
  \rho_d =  (M_d^\nu w_0)(M_d^\nu w_1)\dots(M_d^\nu w_{2^e-1}).
\end{displaymath} 
We take $s=2$ and $N_d=2^{e+d+1}$ and we prove that the sequence
$F_{N_d}(s)$ diverges as $d\rightarrow\infty$.  As we did for $\lambda$, it
is possible to give a lower bound for $F_{N_d}(s)$ by counting occurrences
within $\rho_d$ of words of length $e+d$.

Fix a non-negative integer $d$. By Lemma~\ref{bf:even-invariance'}, there
are two possibilities: either $M_d^\nu$ maps the subspace of even vectors
$\mathbb{P}$ to itself, or it maps it to the set of vectors beginning with
zero.  In the first case, we can replicate essentially verbatim the
procedure we followed for $\lambda$ to get a lower bound for $F_{N_d}(s)$.
In the second case, we have to slightly alter the argument: $z$ is
redefined to be the index of $a$ such that $a_z$ matches the first letter
of $(M_d^\nu w_{n+1})$, and $k$ is redefined to be the number of ones in
$a_{d+1}\dots a_{e+1}$.  Despite these modifications, we reach the same
lower bound for $F_{N_d}(s)$.  Since it diverges as $d\rightarrow\infty$ we
conclude that the sequence of the fractional parts of $(2^n \rho)_{n\geq1}$
does not have Poissonian pair correlations.

\section*{Acknowledgements}
The authors are members of the Laboratoire International Associé SINFIN,
Université Paris Diderot-CNRS/Universidad de Buenos Aires-CONICET).
\bigskip

\bibliographystyle{plain}
\bibliography{paircorr}
\bigskip

{\small
\begin{minipage}{\textwidth}
\noindent
Ver\'onica Becher \\
Departamento de  Computaci\'on,
Facultad de Ciencias Exactas y Naturales \& ICC \\
Universidad de Buenos Aires \&  CONICET, Argentina \\
vbecher@dc.uba.ar
\bigskip\\
Olivier Carton \\
Institut de Recherche en Informatique Fondamentale \\
Universit\'e Paris Diderot, France \\
Olivier.Carton@irif.fr
\bigskip\\
Igmacio Mollo Cunningham\\
Departamento de Matem\'atica\\
Facultad de Ciencias Exactas y Naturales \\
Universidad de Buenos Aires, Argentina \\
imcgham@gmail.com
\end{minipage}
}

\end{document}